\documentclass[a4paper,12pt]{amsart}

\usepackage[utf8]{inputenc}
\usepackage[T1]{fontenc}
\usepackage{mathtools}
\usepackage[left=1.5cm,top=2.5cm,right=1.5cm,bottom=2.5cm]{geometry}

\usepackage{mathrsfs}

\usepackage[usenames,dvipsnames]{color}
\usepackage{setspace}
\usepackage{multicol}
\usepackage{amsthm}
\usepackage{amsmath,amssymb}
\usepackage{enumerate}

\newtheorem{teo}{Theorem}[section] 
\newtheorem*{teo*}{Theorem}
\newtheorem{prop}[teo]{Proposition}

\newtheorem{definicion}{Definition}[section]
\newtheorem{exa}{Example}[section]

\newcommand{\R}{\mathbb{R}}\newcommand{\Rn}{\R^n}\newcommand{\Rm}{\R^m}\newcommand{\Rnn}{\Rmn}\newcommand{\Rmn}{\R^{m\times n}}
\newcommand{\N}{\mathbb{N}}

\renewcommand{\b}{\beta}
\renewcommand{\O}{\Omega}

\renewcommand{\d}{\delta}
\newcommand{\s}{\sigma}

\renewcommand{\l}{\lambda}

\DeclareMathOperator{\cof}{cof}

\usepackage{amsmath}
\usepackage{latexsym}
\usepackage{amssymb}

\title{Bond-based peridynamics does not converge to hyperelasticity as the horizon goes to zero}
\author{J. C. Bellido, J. Cueto, C. Mora-Corral}

\usepackage{esint} 

\begin{document}

\maketitle

\begin{abstract}
Bond-based peridynamics is a nonlocal continuum model in Solid Mechanics in which the energy of a deformation is calculated through a double integral involving pairs of points in the reference and deformed configurations.
It is known how to calculate the $\Gamma$-limit of this model when the horizon (maximum interaction distance between the particles) tends to zero, and the limit turns out to be a (local) vector variational problem defined in a Sobolev space, of the type appearing in (classical) hyperelasticity.
In this paper we impose frame-indifference and isotropy in the model and find that very few hyperelastic functionals are $\Gamma$-limits of the bond-based peridynamics model.
In particular, Mooney-Rivlin materials are not recoverable through this limit procedure.
\end{abstract}

\pagestyle{empty}



\pagestyle{plain}

\section[Introduction]{Introduction}

Peridynamics is a nonlocal continuum model for Solid Mechanics proposed by Silling in \cite{Silling00}. The main difference with classical elasticity (\cite{Ciarlet88}) relies on the nonlocality, reflected in the fact that points separated by a positive distance exert a force upon each other. In this model the use of gradients is avoided by computing internal forces by integration instead of differentiation, and consequently the elastic energy is the result of a double integration. A main feature is that deformations are not assumed to be smooth, in contrast with classical continuum mechanics, where they are required to be, at least, weakly differentiable as functions of a Sobolev space. This makes peridynamics a suitable framework for problems where discontinuities appear naturally, such as fracture, dislocation, or, in general, multiscale materials. Since the pioneering paper \cite{Silling00}, the development of peridynamics has been really overwhelming, both from a theoretical and a numerical-practical point of view. Some references on this are \cite{SiEpWeXuAs07,SiLe10,LeSi08,JaMoOtOt} and the two books \cite{Gerstle, madenci_oterkus}. 

The original model proposed in \cite{Silling00} was the so-called bond-based model, in which the elastic energy is given by a double integral depending on pairs of points in the reference and deformed configurations. It is assumed the existence of a function $w$, named as \emph{pairwise potential function}, such that the total (macroelastic) energy of any deformation $u:\O\rightarrow \Rm$ of the deformable solid $\O\subset \R^n$ is given by 
\begin{equation}\label{eq:nlenergy}
\int_\O\int_{\O\cap B(x,\delta)} w(x-x', u(x)-u(x'))\,dx\,dx',
\end{equation}
where $\delta>0$ is the \emph{horizon}, a model parameter which measures the maximum interaction distance between the particles.
Physically, $n = m = 3$, but it is of mathematical interest to do the analysis for any $n$ and $m$.
In the isotropic linear elastic case, for which the pairwise potential function is quadratic in its second variable, this model limits the Poisson ratio of homogeneous deformations to be $\frac{1}{4}$. This shows that the bond-based model suffers from severe restrictions in order to represent a wide variety of elastic materials. The state-based model was proposed as a more general nonlocal peridynamic model that avoids this serious limitation \cite{SiEpWeXuAs07,SiLe10}. Although the bond-based model presents that restriction, it has been very popular in the last years and has shown to be appropriate and effective in modelling singularity phenomena in situations of practical and academic interest (see the recent survey \cite{JaMoOtOt} and the references therein).  

In this paper, we are concerned with the bond-based model in the general nonlinear situation, and more concretely with its relationship with classical theory of hyperelasticity. In \cite{Silling00}, the link between the peridynamic bond-based model and conventional models was established in terms of the Piola-Kirchhoff stress tensor: given a pairwise potential functional, there exists a stored-energy density function such that the corresponding local Piola-Kirchhoff stress tensor coincides with the nonlocal stress tensor (also measuring force per unit area in the reference configuration). In that work, it is also argued that even though there is a stored-energy density verifying such a condition for a given pairwise potential, the reciprocal is not true, i.e., not for every hyperelastic density one can find a pairwise potential function with the same Piola-Kirchhoff tensor. Our perspective here is different and relies on the energy rather than on the  stress.
First, we recall from the literature in which sense the nonlinear bond-based model \eqref{eq:nlenergy} converges as $\delta \rightarrow 0$ to a local hyperelastic energy like
\begin{equation}\label{eq:lenergy}
\int_\O W(\nabla u(x))\,dx.
\end{equation}
The proper framework to set this issue is $\Gamma$-convergence (\cite{Braides}), or convergence of variational problems, since it infers convergence of infima and minimizers, if they exist. In \cite{BeMCPe}, this question was addressed in a general setting, showing that the $\Gamma$-limit is a vector variational problem, and obtaining an explicit characterization of the $\Gamma$-limit (hence, of the $W$ in \eqref{eq:lenergy}) for a given pairwise potential function $w$. In this investigation, we push forward the calculation in \cite{BeMCPe} aimed to determine whether it is possible or not to recover typical models of hyperelasticity, from bond-based models verifying the natural physical restrictions of frame indifference and isotropy. Our conclusion is that nonlinear bond-based models converge, in the sense of $\Gamma$-convergence, to hyperelastic models with very limited structure and degrees of freedom.
In particular, they cannot converge to a typical hyperelastic model like Mooney-Rivlin.
This result is in agreement with the, previously mentioned, limitations of the bond-based peridynamic model. 

Other references dealing with convergence of peridynamics models to local models as the horizon goes to zero are the following. In the  nonlinear situation, in \cite{SiLe08}, the {\it pointwise} convergence of the state-based peridynamics to classical local models is shown, but the mathematical study in the framework of $\Gamma$-convergence is still pending. For the linear case, in \cite{MeD}, the $\Gamma$-convergence of linear elastic peridynamics to the local Navier-Lam\'e system is shown. This work is extended in \cite{MeD16} to the geometrically nonlinear situation.

The outline of the paper is the following. Section 2 is devoted to preliminaries, including a summary of the $\Gamma$-convergence procedure to pass from nonlocal energies \eqref{eq:nlenergy} to the local energy \eqref{eq:lenergy} as the horizon $\delta$ goes to zero. In Section 3, frame indifference and isotropy are imposed to bond-based models, characterizing the pairwise potential densities giving rise to energies verifying those physical properties. It is also shown that frame indifference and isotropy are preserved when passing to the limit as $\delta\rightarrow 0$.
In Section 4, we perform a preliminar analysis on which stored-energy densities can be recovered when making the $\Gamma$-limit as $\delta \rightarrow 0$ of bond-based models satisfying frame indifference and isotropy.
Finally, in Section \ref{se:MR} we show that Mooney-Rivlin models are not recoverable.
For this, apart from the analysis of the previous section, we need a property of quasiconvexity theory, which states that a strict quasiconvex function can only be the quasiconvexification of itself.

\section{Preliminaries}

\subsection{Nonlinear elasticity}

Let $\O\subset \Rn$ be an open bounded set with smooth boundary. 
This domain represents a solid body before it is deformed. 
A deformed configuration of the body is the image, $u(\O)$, of a (smooth enough) mapping $u:\O\rightarrow \Rm$ called deformation.

An elastic material is called hyperelastic if there exists a function $W:\O\times \Rmn\rightarrow \R$, called stored-energy function of the material, such that the Piola-Kirchhoff stress tensor is the derivative of $W$ with respect to its second variable.
In this case, the \emph{potential elastic} energy of the deformation $u$ is
\begin{equation*}
E(u)=\int_\O W(x,Du(x))\,dx-\int_\O F(x) \cdot u(x)\,dx-\int_{\Gamma_1} g(x)\cdot u(x)\,d \mathcal{H}^{n-1}(x),\end{equation*}
where $F:\O\rightarrow \Rm$ is the distributed body load and $g:\Gamma_1\subset \partial\O\rightarrow \Rm$ the boundary load. The term $d \mathcal{H}^{n-1}$ indicates an $(n-1)$-surface integral.
Moreover, critical points of the energy are equilibirum solutions of Cauchy's equations of motion.

The stored-energy density must satisfy the frame indifference condition 
\begin{equation}\label{eq:localframeindifference}
W(x,RA)=W(x,A)
\end{equation}
for all points $x\in \O$, all matrices $A\in \Rmn$ and all rotations $R\in SO(m)=\{Q\in \R^{m \times m} \colon \det Q = 1 , \, QQ^T=I_m\}$ ($I_m$ the identity matrix of order $m$). This reflects the fact that the deformation energy does not depend on the observer. 

A hyperelastic material is called isotropic if 
\begin{equation}\label{eq:localisotropy} 
W(x,AR)=W(x,A),
\end{equation}
for all points $x\in \O$, all matrices $A\in \Rmn$ and all rotations $R\in SO(n)$. This means that the elastic energy is independent of the stretching, or loading, direction.
When $n = m = 3$, standard representation theorems establish that the hyperelastic body is isotropic if and only if, at each point, the stored-energy function depends only on $|A|^2$, $|\cof A|^2$ and $\det A$; in other words, if and only if there exists $\varphi:\O\times \R\times \R\times \R\rightarrow \R$ such that
\[W(x,A)=\varphi(x, |A|^2,|\cof A|^2, \det A).\]

A paradigmatic example of isotropic hyperelastic materials are Mooney-Rivlin materials, whose stored-energy density is given by the expression
\begin{equation}\label{eq:mr}
W(A)=\alpha |A|^2+\beta |\cof A|^2+g(\det A) , \end{equation}
with $\alpha,\beta>0$ and $g$ a given function. When $\beta=0$, we recover the Neo-Hookean materials. 

An essential reference on the mathematical theory of nonlinear elasticity is \cite{Ciarlet88}. Other references are the books \cite{Antman95,dacorogna,marsden-hughes,Pedregal00} and the survey papers \cite{Ball96,Ball02}.

\subsection{Vector variational problems}

The mathematical theory of hyperleasticity is nowadays a well established subject. In the pioneering work \cite{ball77}, an existence theory in hyperelasticity was given by means of the application of the direct method of the Calculus of Variations to the energy functional 
\begin{equation}\label{eq:localenergy}
I(u)=\int_\O W(x,Du(x))\,dx
.\end{equation}
The direct method has two main ingredients: first, the functional must be coercive, in the sense that $I(u)$ blows up as the norm of $u$ increases.
Second, $I$ must be lower semicontinuous; in our case, weak lower semicontinuous, since the relevant topology in this situation is the weak topology in a Sobolev space. Coercivity is usually guaranteed by imposing proper lower bounds on the stored-energy density $W$. More delicate is the weak lower semicontinuity, which, for integral functionals like $I$, is typically characterized in terms of convexity notions for $W$. In this situation, the relevant convexity concept is quasiconvexity (see \cite{dacorogna} and the references therein). We say that a function $\psi:\Rmn\rightarrow \R$ is \emph{quasiconvex} if and only if
\begin{equation}\label{eq:qc}
\psi(A)\le \int_{(0,1)^n} \psi(A+Dv(x))\,dx, 
\end{equation}
for all matrices $A\in \Rmn$ and test functions $v\in C_c^\infty((0,1)^n,\R^m)$. It turns out that under the standard coercivity and growth conditions, 
\[\frac{1}{C}|A|^p-C\le W(x,A)\le C(1+|A|^p), \qquad \text{a.e. } x\in \O, \quad A\in \Rmn,\]
for some $1<p<\infty$ and $C>0$, and assuming that $W(x,\cdot)$ is quasiconvex for a.e.\ $x\in\O$, the existence of minimizers for $I$ holds. 
In hyperelasticity, quasiconvexity of isotropic hyperelastic stored-energy densities is a consequence of its polyconvexity.
When $n=m=3$, we say that $W$ is polyconvex if it can be expressed as $W(x,A)=\varphi(x,A,\cof A,\det A)$ for some $\varphi$ such that  $\varphi(x,\cdot,\cdot,\cdot)$ is convex for a.e.\ $x\in \O$ (see \cite{dacorogna} for the definition of polyconvexity for general dimensions). Polyconvexity implies quasiconvexity under proper coercivity and growth conditions \cite{ball77,dacorogna}, and, therefore, it is the right convexity notion in this context. When dealing with polyconvex densities, upper bounds can be left behind and existence of minimizers is obtained just by imposing the coercivity conditions
\[\frac{1}{C}(|A|^{p_1}+|\cof A|^{p_2}+|\det A|^{p_3})-C\le W(x,A),\quad  \mbox{a.e. }x\in \O, \quad \mbox{all }A\in \Rmn,\]
for suitable exponents $p_i\ge 1$. 

Well-known references on the application of calculus of variations techniques to nonlinear elasticity are again \cite{Antman95,Ciarlet88,dacorogna,marsden-hughes,Pedregal00}.  

\subsection{Nonlinear bond-based peridynamics}

A general nonlinear energy in the framework of bond-based peridynamics takes the form 
\[E_{nl}(u)=\int_\O\int_\O w(x,x-x',u(x)-u(x'))\,dx'\,dx-\int_\O F(x)\cdot u(x)\,dx,\]
for a given deformation $u:\O\rightarrow \Rm$ \cite{Silling00, SiLe10, madenci_oterkus}. The pairwise potential function $w:\O\times \tilde \O\times \Rm\rightarrow \R$, with $\tilde{\O}=\O-\O$ (set of $x - x'$ with $x, x' \in \O$), measures the interaction between particles $x,x' \in\O$ both in the reference and deformed configurations. As the interaction force between particles is expected to increase as the distance between them decreases, it is natural to assume that the pairwise density $w(x,\cdot,\tilde y)$ blows up at the origin for each $x\in\O$ and $\tilde y\in \Rm$. Furthermore, it is also natural to assume that particles separated by a distance bigger than a parameter $\delta >0$  do not interact at all, so that $w(\cdot,\tilde x,\cdot)=0$ if $|\tilde x|>\delta$. The parameter $\delta$ is called the {\it horizon} of interaction of particles, and it is a relevant part of the nonlocal peridynamic bond-based model. In \cite{BeMC14}, the application of the direct method of the Calculus of Variations for this type of functionals was studied by some of the authors of this paper. An existence theory was obtained in the Lebesgue $L^p$ spaces under suitable growth conditions on $w$, whereas the relevant nonlocal convexity notion requires the function
\begin{equation}\label{eq:nlconvexity}
y \mapsto \int_\O w(x,x-x',y-v(x')\,dx'\end{equation}
to be convex for a.e.\ $x\in\O$ and any test function $v\in L^p(\O,\Rm)$. This condition is actually equivalent, under some technical assumptions, to the weak lower semicontinuity of the functional
\begin{equation}\label{eq:nonlocalenergy}
I_{nl}(u)=\int_\O\int_\O w(x,x-x',u(x)-u(x'))\,dx'\,dx\end{equation}
 in $L^p(\O,\Rm)$ (see \cite{BoElPoSc11,Elbau12}). The study of this nonlocal convexity notion, which is strictly weaker than usual convexity of $w(x,\tilde x,\cdot)$, has been deepened in \cite{BeMC18}, including relaxation of functionals lacking this condition (see also \cite{MoTe19} for the relaxation).

\subsection{Passage from nonlocal to local as the horizon goes to zero}\label{sec:Gammaconvergence}

It is natural to wonder whether the local energy $I$ in \eqref{eq:localenergy} can be recovered as the limit of the nonlocal energy $I_{nl}$ in \eqref{eq:nonlocalenergy} as $\delta\rightarrow 0$. The right framework to study convergence of variational problem is $\Gamma$-convergence \cite{Braides}, as, in particular, it implies convergence of minimizers and minimum energies. The $\Gamma$-convergence of nonlocal functionals $I_{nl}$ as the horizon tends to zero was studied by some of the authors of this paper in \cite{BeMCPe} in an abstract way. It was shown that under natural assumptions the $\Gamma$-limit is a local vector variational problem, and the process to construct such a $\Gamma$-limit was explicitly described. The local $\Gamma$-limit is recovered in several steps:
\begin{enumerate}[i)]
\item {\it Scaling.}\label{step 1} Making explicit the dependence of the nonlocal functional with respect to $\delta$, we include a parameter $\beta$, that will be clarified below, and scale the functional as     
	\begin{equation}\label{eq:Idelta}
     I_\delta(u):=\frac{n+\beta}{\delta^{n+\beta}}
     \int_{\O} \int_{\O \cap B(x,\d)} w(x,x-x',u(x)-u(x'))\,dx'\,dx.
     \end{equation}
\item {\it Blow-up at zero}.\label{step 2} We assume $\beta\in \R$ is such that there exists the limit
     \begin{equation}\label{eq:recension}
     w^\circ(x,\tilde{x},\tilde{y}):= \lim_{t \rightarrow 0} \frac{1}{t^\beta}w(x,t\tilde{x},t\tilde{y}), 
     \end{equation}
     for a.e.\ $x\in\O$, and all $\tilde x\in \tilde{\O}$ and $\tilde y\in \Rm$. 
\item {\it Definition of the local density $\bar{w}$.} We define $\bar{w}:\O\times \Rnn \rightarrow \mathbb{R}$ \label{step 3} as
     \begin{displaymath}
     \bar{w}(x,A):= \int_{\mathbb{S}^{n-1}} w^\circ(x,z,Az) \,d \mathcal{H}^{n-1}(z) \qquad x \in \O, \, A \in \Rnn,
     \end{displaymath}
     where $\mathbb{S}^{n-1}$ is the $n$-dimensional unit sphere.
\item {\it Quasiconvexification.}\label{step 4} The candidate a $\Gamma$-limit of $I_\delta$ as $\delta \rightarrow 0$ is then
     \begin{displaymath}
     \tilde{I}(u):=\int_{\O}\bar{w}^{qc}(x,Du (x))\,dx,
     \end{displaymath}
     where $w^{qc}(x,\cdot)$ is the quasiconvexification (\cite{dacorogna}) of $\bar{w}(x,\cdot)$ defined as
     \[\bar{w}^{qc}(x,A):=\sup\{ v(x,A): v(x,\cdot) \leq \bar{w}(x,\cdot) \text{ and } v(x,\cdot) \text{ quasiconvex} \}.\]
     
\end{enumerate}
Under several assumptions on the pairwise potential density $w$, including the nonlocal convexity \eqref{eq:nlconvexity}, it is shown in \cite{BeMCPe} that $I_\delta$ $\Gamma$-converges to $\tilde I$.

Our aim in what follows is to check whether typical stored-energy densities in hyperelasticity can be obtained by this procedure.
This amounts to asking whether hyperelastic models can be obtained as the $\Gamma$-limit of bond-based peridynamics models as the horizon of interaction of particles goes to zero.
More concretely, given a polyconvex stored-energy density $W$, whether there exists a pairwise potential function $w$ such that its corresponding functional $I_\delta$ $\Gamma$-converges to $I$ as $\delta\rightarrow 0$.

\section{Frame-indifference and isotropy in the bond-based model}

In this section we explore how frame indifference and isotropy are translated in mathematical terms into the nonlinear bond-based model.
Frame indifference requires that
\begin{equation}\label{eq:frameindifference}
w(x,\tilde{x},\tilde{y})=w(x,\tilde{x},R\tilde{y}), \qquad \text{a.e. } x\in \O, \quad \tilde{x}\in\tilde{\O}, \quad \tilde{y}\in \Rm , \quad R\in SO(m) .
\end{equation}
We are also interested in isotropic materials, i.e., those whose deformation energy does not depend on the loading, or stretching, direction. Mathematically, this is imposed in the pairwise potential function by requiring
\begin{equation}\label{eq:isotropy}
w(x,\tilde{x},\tilde{y})=w(x,R\tilde{x},\tilde{y}), \qquad \text{a.e. } x\in \O, \quad \tilde{x}\in\tilde{\O}, \quad \tilde{y}\in \Rm , \quad R\in SO(n) .
\end{equation}
The next result, which has appeared in the literature before (for instance in \cite{Silling00,SiLe10}), is straightforward.

\begin{prop}\label{prop:characterization w} The bond-based model satisfies frame indifference and isotropy (i.e., the pairwise potential function $w$ satisfies \eqref{eq:frameindifference} and \eqref{eq:isotropy}) if and only if there exists $\tilde{w}:\O\times  \R\times \R\rightarrow\R$ such that 
\[w(x,\tilde{x},\tilde{y})=\tilde{w}(x,|\tilde{x}|,|\tilde{y}|) \qquad \text{a.e. } x\in \O, \quad \tilde{x}\in\tilde{\O}, \quad \tilde{y}\in \Rm . \]
\end{prop}

An interesting question is whether frame indifference and isotropy are inferred to the $\Gamma$-limit obtained as the horizon goes to zero. The answer to this question is positive, as the next result shows. 

\begin{prop}\label{prop:passing to the limit in the physical properties}
Given a pairwise potential function $w$, assume there exists $\beta\in\R$ such that the limit in \eqref{eq:recension} exists and the function $w^\circ$ may be defined. If the pairwise potential function $w$ satisfies \eqref{eq:frameindifference} and \eqref{eq:isotropy}, then the function $W$, obtained from $w$ by the procedure described in Section \ref{sec:Gammaconvergence} ($W=\bar{w}^{qc}$), satisfies \eqref{eq:localframeindifference} (frame indiference) and \eqref{eq:localisotropy} (isotropy).
\end{prop}
\begin{proof}
We prove frame indifference and isotropy of $W$ all at once, but we emphasize that any of those properties of $W$ is inferred independently from the corresponding property of $w$. By Proposition \ref{prop:characterization w}, there exists $\tilde w$ such that
\[w(x,\tilde{x},\tilde{y})=\tilde{w}(x,|\tilde{x}|,|\tilde{y}|), \qquad \text{a.e. } x\in \O, \quad \tilde{x}\in\tilde{\O}, \quad \tilde{y}\in \Rm .\]
By assumption, there exists $\beta\in\R$ such that 
\[w^\circ(x,\tilde{x},\tilde{y}):= \lim_{t \rightarrow 0} \frac{1}{t^\beta}w(x,t\tilde{x},t\tilde{y}), \qquad \text{a.e. } x\in \O, \quad \tilde{x}\in\tilde{\O}, \quad \tilde{y}\in \Rm . \]     
Then 
\[w^\circ(x,\tilde{x},\tilde{y})=\lim_{t\rightarrow 0} \frac{1}{t^\beta} \tilde{w}(x,t|\tilde{x}|,t|\tilde{y}|), \qquad \text{a.e. } x\in \O, \quad \tilde{x}\in\tilde{\O}, \quad \tilde{y}\in \Rm ,\]
and we write, with a small abuse of language, that 
\[w^\circ=w^\circ(x,|\tilde{x}|,|\tilde{y}|).\]
Given two rotations $R_1 \in SO(m)$ and $R_2\in SO(n)$, 
\begin{align*}
\bar{w}(x,R_1AR_2) & =\int_{\mathbb{S}^{n-1}} w^\circ(x,|z|,|R_1AR_2z|)\,d\mathcal{H}^{n-1}(z) \\
&=\int_{\mathbb{S}^{n-1}} w^\circ(x,|z|,|AR_2z|)\,d\mathcal{H}^{n-1}(z) \\
&= \int_{\mathbb{S}^{n-1}} w^\circ(x,|R_2^{-1}z|,|Az|)\,d\mathcal{H}^{n-1}(z) \\
&=\int_{\mathbb{S}^{n-1}} w^\circ(x,|z|,|Az|)\,d\mathcal{H}^{n-1}(z) \\
&=\bar{w}(x,A),\end{align*}
hence $\bar w$ satisfies \eqref{eq:localframeindifference} and \eqref{eq:localisotropy}, and, therefore, by \cite[Th.\ 6.14]{dacorogna}, so does its quasiconvexification $\bar{w}^{qc}=W$. 
\end{proof}

\section{Which local densities can be recovered from nonlocal ones?}

In this section we make it explicit the relationship between a pairwise potential function $w$ and the stored-energy function obtained from it by the previous $\Gamma$-convergence procedure, in the presence of frame indifference and isotropy.
We obtain an explicit identity involving those two functions, in such a way that given a stored-energy function $W$, any candidate $w$ to pairwise potential function such that the corresponding sequence of nonlocal functionals $\Gamma$-converges to the local functional given by $W$ must satisfy. Thus, it provides a criterium in order to check whether a local functional with energy density $W$ may be obtained as the $\Gamma$-limit of functionals like \eqref{eq:Idelta}. There is no doubt this is interesting from a mathematical point of view, but also from a mechanical perspective, since it permits to answer whether local hyperelastic energies are the $\Gamma$-limit of nonlocal bond-based nonlinear models as the horizon goes to zero. 

For the next result we consider a pairwise potential function $w$ verifying the frame indifference and isotropy properties, and for simplicity in the exposition we assume that the material is homogeneous, i.e., $w$ does not depend on the material point $x$, so that $w=w(\tilde{x},\tilde{y})$. By Proposition \ref{prop:characterization w}, there exists $\tilde{w}: \R\times\R\rightarrow \R$ such that 
\[w(\tilde{x},\tilde{y})=\tilde{w}(|\tilde x|,|\tilde y|),\qquad  \tilde{x}\in\tilde{\O}, \quad \tilde{y}\in \Rm.\]
Note that step \ref{step 2}) of Section \ref{sec:Gammaconvergence} makes $w^{\circ}$ a homogeneous function of degree $\b$ in the pair $(\tilde{x}, \tilde{y})$, i.e.,  $w^{\circ} (x, t \tilde{x}, t \tilde{y}) = t^{\beta} w^{\circ} (x, \tilde{x}, \tilde{y})$.
Therefore, we may assume, without loss of generality, that $w$ is itself homogeneous of degree $\beta$, which amounts to saying that $w=w^\circ$, with $w^\circ$ defined by \eqref{eq:recension}. Then, according to Section \ref{sec:Gammaconvergence}, the density of the $\Gamma$-limit of $I_\delta$ is the quasiconvexification of the function
\[
 \bar{w}(A) =\int_{\mathbb{S}^{n-1}} w(z,Az)\,d\mathcal{H}^{n-1}(z)=\int_{\mathbb{S}^{n-1}} \tilde{w}(|z|,|Az|)\,d\mathcal{H}^{n-1}(z) =\int_{\mathbb{S}^{n-1}} \tilde{w}(1,|Az|)\,d\mathcal{H}^{n-1}(z).
\]
Notice that the dependence of $w$ on $\tilde x$ is irrelevant in order to obtain $\bar w$.

The above process motivates the following definition of recoverable function.

\begin{definicion}
The function $W : \Rmn \to \R$ is recoverable if there exist $\bar{w} : \Rmn \to \R$ and $\tilde{w} : \{1\} \times [0, \infty) \to \R$
such that
\begin{equation}\label{eq:recover}
  \bar{w}(A) = \int_{\mathbb{S}^{n-1}} \tilde{w}(1,|Az|)\,d\mathcal{H}^{n-1} (z), \qquad A \in \Rmn
\end{equation}
and $W = \bar{w}^{qc}$.
\end{definicion}

The following result gives a necessary and sufficient condition for a $\bar{w}$ to satisfy \eqref{eq:recover} without invoking $\tilde{w}$.
We denote by $\s_{n-1}$ the $\mathcal{H}^{n-1}$ measure of $\mathbb{S}^{n-1}$, and by $\fint$ the mean integral.

\begin{prop}\label{pr:recover}
The function $\bar{w}$ satisfies \eqref{eq:recover} if and only if for every $A \in \Rnn$ one has
\begin{equation}\label{eq:wrecover}
 \bar{w}(A) = \fint_{\mathbb{S}^{n-1}} \bar{w} (|Az| I )\,d\mathcal{H}^{n-1}(z) .
\end{equation}
In this case, a function $\tilde{w}$ giving rise to \eqref{eq:recover} is 
\begin{equation}\label{eq:tildewdef}
 \tilde{w} (1, t) = \frac{\bar{w} (t I)}{\s_{n-1}} , \qquad t \geq 0 .
\end{equation}
\end{prop}
\begin{proof}
Assume $\bar{w}$ satisfies \eqref{eq:recover} for some $\tilde{w}$.
Fixed $z \in \mathbb{S}^{n-1}$, by \eqref{eq:recover} we have
\[
 \bar{w}(|Az| I) = \int_{\mathbb{S}^{n-1}} \tilde{w}(1,  \left| |Az| I z' \right|)\,d\mathcal{H}^{n-1}(z') = \int_{\mathbb{S}^{n-1}} \tilde{w}(1,  |Az| )\,d\mathcal{H}^{n-1}(z') = \s_{n-1} \tilde{w}(1,  |Az| ) ,
\]
hence, combining this with \eqref{eq:recover} we obtain \eqref{eq:wrecover}.

Conversely, assuming that \eqref{eq:wrecover} holds we define $\tilde{w}$ as \eqref{eq:tildewdef} and we readily obtain \eqref{eq:recover}.
\end{proof}

What formula \eqref{eq:wrecover} says is that $\bar{w}$ is determined just by its values in matrices multiples of the identity.

We implement Proposition \ref{pr:recover} to find several examples of stored-energy functions that come or do not come from a $\tilde{w}$ as in \eqref{eq:recover}. 

\begin{exa} The functional 
\[I_\delta(u)=\frac{n}{\delta^n}\int_\O\int_\O \frac{n}{\sigma_{n-1}} \frac{|u(x)-u(x')|^2}{|x-x'|^2}\,dx'\,dx,\quad u\in L^2(\O,\Rm)\]
$\Gamma$-converges as $\delta\rightarrow 0$ to 
\[I(u)=\int_\O|\nabla u(x)|^2\,dx,\quad u\in H^1(\O, \Rm).\]

This assertion is justified by the result of Section \ref {sec:Gammaconvergence} (noting that $\beta=0$ in this case), formula \eqref{eq:wrecover} and the following simple computation:
\begin{align*}
\bar{w}(A) & =\int_{\mathbb{S}^{n-1}} \frac{n}{\sigma_{n-1}} |Az|^2\, d\mathcal{H}^{n-1}(z)={n} \sum_{i=1}^m \fint_{\mathbb{S}^{n-1}} \left( \sum_{j=1}^n A_{ij}z_j\right)^2\,d\mathcal{H}^{n-1}(z)\\
& ={n} \sum_{i=1}^m \sum_{j,k=1}^n\fint_{\mathbb{S}^{n-1}}  A_{ij}A_{ik}z_j z_k\, d\mathcal{H}^{n-1}(z)=  \sum_{i=1}^m \sum_{j=1}^n A_{ij}^2\\
&= |A|^2 .
\end{align*}
\end{exa}

The following result shows in particular that $|A|^p$ does not satisfy \eqref{eq:wrecover} for $p \notin \{ 0, 2 \}$.
Indeed, as a consequence of the previous and next examples we have that the only functions of $|A|^2$ that satisfy \eqref{eq:wrecover} are affine functions of $|A|^2$. 

\begin{exa}\label{ex:A}
Let $g : [0, \infty) \to \R$ be a strictly convex, or strictly concave, function on a subinterval $I\subset [0,+\infty)$.
Then the function $g (|A|^2)$ does not satisfy \eqref{eq:wrecover}.

Indeed, assume, for a contradiction, that formula \eqref{eq:wrecover} holds.
Then,
\[
 g (|A|^2) = \fint_{\mathbb{S}^{n-1}} g \left( \left| |Az| I \right|^2 \right) d\mathcal{H}^{n-1}(z) = \fint_{\mathbb{S}^{n-1}} g \left( n |Az|^2 \right) d\mathcal{H}^{n-1}(z) .
\]
Now, we take any $A$ such that the set $\{|Az| \colon z \in \mathbb{S}^{n-1}\}$ contains more than one point, and such that $\{ n|A z|^2 \colon z \in \mathbb{S}^{n-1} \}\subset I$. Then, by Jensen's inequality, using that $g$ is strictly convex, we obtain
\[
 \fint_{\mathbb{S}^{n-1}} g \left( n |Az|^2 \right) d\mathcal{H}^{n-1}(z) > g \left( n \fint_{\mathbb{S}^{n-1}} |Az|^2 \, d\mathcal{H}^{n-1}(z) \right) = g (|A|^2) ,
\]
which is a contradiction.
If $g$ is strictly concave, the inequality above is reversed and we also obtain a contradiction.
\end{exa}

The following example shows in particular that $|\cof A|^p$ does not satisfy \eqref{eq:wrecover} for $p \geq 1$.

\begin{exa}\label{ex:cofA}
Fix $n=m=3$.
Let $g: [0, \infty) \to \R$ be a convex function such that $\limsup_{t \to \infty} g(t) = \infty$.
Then the function $\bar{w}(A)=g (|\cof A|)$ does not satisfy \eqref{eq:wrecover}. 

Indeed, given $A \in \R^{3 \times 3}$ and $z \in \mathbb{S}^2$ we have that $\left| \cof (|Az| I) \right| = \sqrt{3} |Az|^2$, so if formula \eqref{eq:wrecover} holds then, by Jensen's inequality,
\[
 g (|\cof A|) = \fint_{\mathbb{S}^2} g \left( \sqrt{3} |Az|^2 \right) d\mathcal{H}^2 (z) \geq g \left( \sqrt{3} \fint_{\mathbb{S}^2 } |Az|^2 \, d\mathcal{H}^2 (z) \right) = g \left( \frac{1}{\sqrt{3}} |A|^2 \right) .
\]
Now fix $\l >0$ and consider $A$ as the matrix with diagonal elements $\l, \l^{-1}, 1$.
Then, the inequality
\[
g (|\cof A|) \geq g \left( 3^{-1/2} |A|^2 \right)
\]
reads as
\[
g \left( \sqrt{\l^2 + \l^{-2} + 1} \right) \geq g \left( 3^{-1/2} (\l^2 + \l^{-2} + 1) \right) ,
\]
which amounts to saying that
\[
g \left( t \right) \geq g \left( 3^{-1/2} t^2 \right) , \qquad t \geq \sqrt{3} . 
\]
Fix $t_0 \geq \sqrt{3}$.
Then
\[
 \max_{[t_0, 3^{-1/2} t_0^2]} g \geq \max_{[3^{-1/2} t_0^2, 3^{-3/2} t_0^4]} g .
\]
Repeating this argument and applying induction we find that
\[
 \max_{[t_0, 3^{-1/2} t_0^2]} g \geq \max_{\left[ t_0, \sqrt{3} \left( \frac{t_0}{\sqrt{3}} \right)^{2^n} \right]} g
\]
for all $n \in \N$.
Taking $t_0 = 2 \sqrt{3}$ we obtain that
\[
 \max_{[2 \sqrt{3}, 4 \sqrt{3}]} g \geq \max_{\left[ 2 \sqrt{3}, 2^{2^n} \sqrt{3} \right]} g .
\]
Consequently, 
\[
 \max_{[2 \sqrt{3}, 4 \sqrt{3}]} g \geq \sup_{\left[ 2 \sqrt{3}, \infty \right)} g ,
\]
which contradicts the assumption $\limsup_{t \to \infty} g(t) = \infty$.
\end{exa}
In Example \ref{ex:cofA}, the convexity hypothesis on $g$ may be relaxed to $g$ being convex on an interval $[a,\infty)$ for some $a>0$.

A similar reasoning can be done with $g (\det A)$, but we postpone to the next section a more definitive result.

\section{Mooney-Rivlin materials are not recoverable}\label{se:MR}

By small adaptations of the arguments of Examples \ref{ex:A} and \ref{ex:cofA}, one can exhibit large families of stored-energy functions $\bar{w}$ that do not satisfy \eqref{eq:wrecover}.
Those examples by themselves do not prove that they are not recoverable.
Without the aim of being exhaustive, we present in this section the fact that Mooney-Rivlin materials are not recoverable.
In order to do that, we will use the following sufficient condition for which equality $W = \bar{w}$ in the procedure of Section \ref{sec:Gammaconvergence} holds.
This result is possibly known for experts in quasiconvexity, but we have not found a reference of it.
First we need the definition of strict quasiconvexity.
We say that a function $\psi:\Rmn\rightarrow \R$ is \emph{strictly quasiconvex} if and only if
\begin{equation}\label{eq:qc}
\psi(A) < \int_{(0,1)^n} \psi(A+Dv(x))\,dx, 
\end{equation}
for all matrices $A\in \Rmn$ and test functions $v\in C_c^\infty((0,1)^n,\R^m) \setminus \{ 0 \}$.

The definition of strict polyconvexity is as follows. We say that $\psi: \Rmn \rightarrow \R$ is \emph{strictly polyconvex} if there exists a strictly convex function $g$ defined in the set of minors of $\Rmn$ matrices such that $\psi (A) = g (M(A))$ for all $A \in \Rmn$, where $M(A)$ is the vector formed by all minors of the matrix $A$ is a given order.

The following sufficient condition for strict quasiconvexity is useful.

\begin{prop}\label{pr:spcsqc}
Any strictly polyconvex function is strictly quasiconvex.
\end{prop}
\begin{proof}
Let $\psi: \Rmn \rightarrow \R$ be strictly polyconvex, and let $g$ be strictly convex such that $\psi (A) = g (M(A))$ for all $A \in \Rmn$.
Fix $A \in \Rmn$ and $v\in C^\infty_c((0,1)^n,\R^m) \setminus \{0\}$.
First of all, since $M$ is quasiaffine then, by a well known result \cite{dacorogna},
\[\int_{(0,1)^n} M(A+D v(x))\,dx=M(A) .\]
Now, applying Jensen's inequality, and having into account that $g$ is strictly convex, we have that 
\begin{align*}
\int_{(0,1)^n} g(M(A+D v(x)))\,dx & >g\left( \int_{(0,1)^n} M(A+D v(x))\,dx \right) =g(M(A)) .
\end{align*}
Consequently, $\psi$ is strictly quasiconvex. 
\end{proof}

The result that we seek guaranteeing the equality $W = \bar{w}$ is the following.

\begin{prop}\label{pr:sqc}
Let $W : \Rmn \to \R$ be strictly quasiconvex and let $\bar{w} : \Rmn \to \R$ be such that $W = \bar{w}^{qc}$.
Then $W = \bar{w}$.
\end{prop}
\begin{proof} Our proof is based on the application of gradient Young measures \cite{Pedregal97}. Let $\mathcal{Y}(A)$ be the set of homogeneous gradient Young measures of barycenter $A$.
It is known that  $W$ is strictly quasiconvex if and only if 
\[W(A)< \int_{\Rmn} W(F)\, d\nu(F),\]
for all $A\in \Rmn$ and all $\nu\in \mathcal{Y}(A)\setminus \{\delta_A\}$, where $\delta_A\in \mathcal{Y}(A)$ is the Dirac delta at $A$.

By assumption, $W = \bar{w}^{qc} \leq \bar{w}$.
Fix $A \in \Rmn$.
By the expression of the quasiconvexification in terms of Young measures (\cite{Pedregal97}), there exists $\nu \in \mathcal{Y}(A)$, such that 
\[\bar{w}^{qc}(A) = \int_{\Rmn} \bar{w}(F)\,d\nu(F) .\]
If $\nu \neq \d_A$ then
\[
 \int_{\Rmn} \bar{w}(F)\,d\nu(F) \geq \int_{\Rmn} W (F)\,d\nu(F) > W(A) ,
\]
a contradiction with the fact $\bar{w}^{qc} =W$.
Therefore, $\nu = \d_A$ and, hence, $\bar{w}^{qc} (A) = \bar{w} (A)$.
\end{proof}

With all those preliminaries, we are in a position to prove the final result of this paper.

\begin{prop}\label{pr:MR}
Let $n = m = 3$.
Let $g : (0, \infty) \to [0, \infty)$ be convex and such that there exists $a >0$ for which $g |_{[a, \infty)}$ is increasing.
Let $\alpha, \beta \geq 0$.
Assume that:
\[
 \alpha>0 \quad \text{or} \quad \beta > 0 \quad \text{or} \quad g \text{ is strictly convex.}
\]
Suppose, in addition, that
\begin{equation}\label{eq:beta0}
 \text{if } \beta = 0, \text{ for all } t \geq a \text{ there exists } t_1 >t \text{ with } g(t_1) > g(t) .
\end{equation}
Then the function $W$ of \eqref{eq:mr} is not recoverable.
\end{prop}
\begin{proof}
Assume, for a contradiction, that $W$ is recoverable.
We use the notation of Section \ref{sec:Gammaconvergence}.
The assumptions of $\alpha, \beta$ and $g$ imply that $W$ is strictly polyconvex.
By Proposition \ref{pr:spcsqc}, it is strictly quasiconvex, and, in turn, by Proposition \ref{pr:sqc} we have that $W = \bar{w}$.

The proof wil be finished as soon as we show that $\bar{w}$ does not satisfy \eqref{eq:wrecover}.
For any $A \in \R^{3 \times 3}$, using Jensen's inequality we find that
\begin{equation}\label{eq:A4}
 \fint_{\mathbb{S}^2} |Az|^4 \, d \mathcal{H}^2 (z) \geq \left( \fint_{\mathbb{S}^2} |Az|^2 \, d \mathcal{H}^2 (z) \right)^2 = \frac{|A|^4}{9} .
\end{equation}
Now, it is easy to check that the expression
\[
 \left( \fint_{\mathbb{S}^2} |Az|^3 \, d \mathcal{H}^2 (z) \right)^{\frac{1}{3}}
\]
defines a norm in $\R^{3 \times 3}$.
Since all norms are equivalent in $\R^{3 \times 3}$, there exists $c >0$ such that
\[
 \fint_{\mathbb{S}^2} |Az|^3 \, d \mathcal{H}^2 (z) \geq c |A|^3 , \qquad A \in \R^{3 \times 3} .
\]
In fact, we can assume that $c \leq a^{-2}$.
Using Jensen's inequality, we find that for all $A \in \R^{3 \times 3}$ with $c |A|^3 \geq a$,
\begin{equation}\label{eq:gA3}
  \fint_{\mathbb{S}^2} g (|Az|^3 ) \, d \mathcal{H}^2 (z) \geq g \left( \fint_{\mathbb{S}^2} |Az|^3 \, d \mathcal{H}^2 (z) \right) \geq g (c |A|^3) .
\end{equation}
Using \eqref{eq:A4} and \eqref{eq:gA3}, we have that for all $A \in \R^{3 \times 3}$ with $c |A|^3 \geq a$,
\[
 \fint_{\mathbb{S}^2} \bar{w} (|Az| I)  \, d \mathcal{H}^2 (z) \geq \alpha |A|^2 + \frac{\beta}{3} |A|^4 + g (c |A|^3) . 
\]
If $\bar{w}$ were recoverable, we would have that, if $c |A|^3 \geq a$,
\[
 \beta |\cof A|^2 + g (\det A) \geq \frac{\beta}{3} |A|^4 + g (c |A|^3) .
\]
Let $\lambda >0$ and let $A$ be the diagonal matrix with diagonal elements $\lambda, 1/\lambda, \sqrt[3]{a/c}$.
Then, the inequality above reads as
\[
 \beta \left( 1 + \lambda^2 \left(\frac{a}{c}\right)^{\frac{2}{3}} + \frac{1}{\lambda^2} \left(\frac{a}{c}\right)^{\frac{2}{3}} \right) + g \left( \left(\frac{a}{c}\right)^{\frac{1}{3}} \right) \geq \frac{\beta}{3} \left( \lambda^2 + \frac{1}{\lambda^2} + \left(\frac{a}{c}\right)^{\frac{2}{3}} \right)^2 + g \left( c \left( \lambda^2 + \frac{1}{\lambda^2} + \left(\frac{a}{c}\right)^{\frac{2}{3}} \right)^\frac{3}{2} \right) .
\]
If $\beta >0$ then 
\[
 \beta \left( 1 + \lambda^2 \left(\frac{a}{c}\right)^{\frac{2}{3}} + \frac{1}{\lambda^2} \left(\frac{a}{c}\right)^{\frac{2}{3}} \right) + g \left( \left(\frac{a}{c}\right)^{\frac{1}{3}} \right) \geq \frac{\beta}{3} \left( \lambda^2 + \frac{1}{\lambda^2} + \left(\frac{a}{c}\right)^{\frac{2}{3}} \right)^2 ,
\]
which yields a contradiction when we send $\lambda \to \infty$.
If $\beta = 0$ we obtain 
\[
  g \left( \left(\frac{a}{c}\right)^{\frac{1}{3}} \right) \geq g \left( c \left( \lambda^2 + \frac{1}{\lambda^2} + \left(\frac{a}{c}\right)^{\frac{2}{3}} \right)^\frac{3}{2} \right) ,
\]
which is again a contradiction due to \eqref{eq:beta0}, since $g$ is increasing in $[a, \infty)$, and
\[
 a \leq \left(\frac{a}{c}\right)^{\frac{1}{3}} \leq c \left( \lambda^2 + \frac{1}{\lambda^2} + \left(\frac{a}{c}\right)^{\frac{2}{3}} \right)^\frac{3}{2} 
\]
provided that $\lambda$ is large enough.
\end{proof}

An analogue result holds in the incompressible case.

\begin{prop}
Let $n = m = 3$.
Let $\alpha, \beta \geq 0$.
Then the function $W : \R^{3 \times 3} \to \R \cup \{ \infty\}$ defined by
\[
 W(A)= \begin{cases}
 \alpha |A|^2+\beta |\cof A|^2 , & \text{if } \det A = 1 , \\
 \infty & \text{if } \det A \neq 1 
 \end{cases}
\]
is not recoverable.
\end{prop}
\begin{proof}
Assume, for a contradiction, that $W$ is recoverable.
It is easy to check that the function $g : \R \to \R \cup \{ \infty \}$ defined by $g(t) = \infty$ for $t \neq 1$ and $g(1) = 0$ is strictly convex at $t=1$. 
As in Proposition \ref{pr:MR}, $W$ is strictly polyconvex where it is finite, so by Propositions \ref{pr:spcsqc} and \ref{pr:sqc}, $W = \bar{w}$.

If $\bar{w}$ were not recoverable, by \eqref{eq:wrecover} we reach a contradiction by considering the matrix $A$ with diagonal elements $\l, 1/\l, 1$ with $\l >1$, since the right-hand side is infinity, while the left-hand side is finite.
\end{proof}

Using the ideas of Example \ref{ex:A}, one can generalize Proposition \ref{pr:MR} to rule out the recoverability of many families of functions of the style of Mooney--Rivlin, but replacing $|\cof A|^2$ with another convex function of $\cof A$.
However, for the sake of simplicity, we have restricted ourselves to a quadratic dependence on $\cof A$.


\section*{Acknowledgements}
This work has been supported by the Spanish Ministry of Economy and Competitivity through projects MTM2017-83740-P (J.C.B.\ and J.C.), and MTM2017-85934-C3-2-P (C.M.-C.).

\end{document}